\documentclass[12pt]{amsart}
\usepackage{mathtools}
\usepackage{amsmath,amsxtra,amssymb,latexsym, amscd,amsthm,times, fullpage, float}
\usepackage{indentfirst}
\usepackage{url}
\usepackage[mathscr]{eucal}
\restylefloat{table}
\usepackage{hyperref}
\def\bb{\mathbb}

\newtheorem{Lemma}{Lemma}[section]
\newtheorem{Th}[Lemma]{Theorem}

\newtheorem{Prop}[Lemma]{Proposition}

\newtheorem{Cor}[Lemma]{Corollary}
\newtheorem*{Def}{Definition}

\newtheorem*{theorem}{Main Theorem}
\newtheorem*{corollary}{Corollary}
\newtheorem*{corollary 2}{Corollary 2}
\newtheorem{conjecture}{Conjecture}

\begin{document}
\title[Normal Subgroups of Division Rings]{On normal subgroups of division rings which are \\radical over a proper division subring}
\dedicatory{Dedicated to Professor Bui Xuan Hai on his 60th birthday}

\author{Mai Hoang Bien}
\email{maihoangbien012@yahoo.com}
\address{Department of Basic Sciences, University of Architecture, 196 Pasteur Street, District 1, HCM City, Vietnam.}

\author{Duong Hoang Dung}
\email{dhdung1309@gmail.com, dhoang@math.uni-bielefeld.de}
\address{Fakult\"{a}t f\"{u}r Mathematik, Universit\"{a}t Bielefeld,
Postfach 100131, D-33615 Bielefeld, Germany.}

 \keywords{Divsion ring, Kurosh element, centrally finite, locally finite. \\ \protect \indent 2010 {\it Mathematics Subject Classification.} 16K20, 16K40.}
\begin{abstract} 
We introduce Kurosh elements in division rings based on the idea of a conjecture of Kurosh.
Using this, we generalize a result of Faith in \cite{Faith} and of Herstein in \cite{Herstein1}.
\end{abstract}
\maketitle

\section{Introduction}\label{1}
The structure properties of multiplicative subgroups  in division rings have been  recently studied  such as free subgroups (\cite{GT12, MM11}), maximal subgroups (\cite{ DBH12,KM10, KM09, M02}), subgroups radical over a set (\cite{AEM03,DBH12, Herstein1,MA96, MM11,M00}), etc. In this paper, we generalize one of Faith's works on division rings which are radical over a proper subring. The best well known result concerning the radicality in division rings is a result of Kaplansky, saying that every division ring which is radical over its center is in fact commutative (\cite[Theorem 15.15]{Lam-first-course}). Faith generalized this result by proving in \cite{Faith} that  every division ring which is radical over a proper subring is also commutative. Here, an element $x$ of a division ring $D$ is  {\it radical} over a subring $K$ of $D$ if there exists a positive integer $n_x$ such that $x^{n_x}\in K$. A subset of $D$ is  {\it radical} over $K$ if all of its elements are radical over $K$.  In group-theoretical language, Faith's Theorem said that if the multiplicative group $D^*$ is radical over $K$ then $D^*$ is commutative. The goal of this paper is to consider whether the result holds when $D^*$ is replaced by a normal subgroup $N$ of $D^*$; that is, whether it is true that if a normal subgroup $N$ of $D^*$  is radical over a proper subring $K$ then $N$ is commutative. Unfortunately, in general, if $K$ is not a division subring, then it fails to be true. In the end of Section 4, we provide an example to illustrate this. We formulate the following conjecture.
\begin{conjecture} \label{Bien conjecture}
Let $D$ be a division ring and $K$ a proper division subring of $D$. Then every normal subgroup of $D^*$ radical over $K$ is commutative.
\end{conjecture}

In particular, when $K$ is the center of $D$, then Conjecture \ref{Bien conjecture} is a conjecture of Herstein in \cite{Herstein1}:
\begin{conjecture}\label{Herstein}\cite{Herstein1} Let $D$ be a division ring with center $F$ and $N$ a normal subgroup of $D^*$. If $N$ is radical over $F$ then $N\subseteq F$.
\end{conjecture}
In fact, Conjecture \ref{Herstein} was orginally stated in \cite{Herstein1} for subnormal subgroups $N$ of $D^*$. However, subnormal subgroups can be replaced by normal ones by \cite[Lemma 1]{Herstein2} and because that if $N$ is normal and commutative in $D^*$ then $N$ is contained in the center of $D$ \cite[14.4.4]{S88}. Conjecture \ref{Herstein} is still open in general. We would like to study Conjecture \ref{Bien conjecture} for the general case. Our main techniques come from Kurosh's Conjecture for division rings.

Let $D$ be a division ring with center $F$. If $D$ is a finite dimensional vector space over $F$ then we say that $D$ is {\it centrally finite} (see \cite[Defnition 14.1]{Lam-first-course}). An element $a$ in $D$ is {\it non-central} if $a\notin F$. It is {\it algebraic} (over the center $F$) if it is a root of some polynomials over $F$. If all elements of $D$ are algebraic then we say that $D$ is {\it algebraic} (over the center $F$). One says $D$ {\it locally finite} (see \cite{Zelmanov}) if the division subring $F(S)$ generated by $F$ and a finite subset $S$ of $D$ is a finite dimensional vector space over $F$. Hence, the class of locally finite division rings contains the class of centrally finite ones and is contained in the class of algebraic ones. There are examples of locally finite division rings which are not centrally finite \cite{Lam-first-course}. In 1941, Kurosh conjectured (see \cite{Kurosh41,Zelmanov}) that every algebraic division ring is locally finite. In general the Kurosh conjecture is still open.  Assume that $D$ is algebraic and if this conjecture of Kurosh holds, then for any non-central element $x$ of $D$, there exists a division subring of $D$ containing $x$ as a non-central element. Using this idea, we introduce the following notion.

\begin{Def}
Let $D$ be a division ring. A non-central element $x$ of $D$ is {\it  Kurosh}  if there exists a centrally finite division subring of $D$ containing $x$ as a non-central element.
\end{Def} 

Note that if Kurosh's Conjecture holds then every non-central element of an algebraic division ring is a Kurosh element. In particular, every non-central element of a locally finite division ring is Kurosh. In Section 2, some classes of Kurosh elements in an arbitrary division ring will be described. Notice also that in the definition, one does not require $D$ to be algebraic. Moreover, if every non-central element of $D$ is Kurosh, it does not mean that $D$ is algebraic. In the last section of this paper, using the Mal'cev-Neumann's construction of Laurent series rings, we present an example of a division ring $D$ with the properties that all elements of $D$ are Kurosh but $D$ is not algebraic. The main result in this paper is the following.
\begin{theorem}[Theorem 3.4]
Let $D$ be a division ring with center $F$, $K$ a proper division subring of $D$ containing $F$ and $N$ a normal subgroup of $D^*$. If $N$ is radical over $K$ then $N\setminus K$ does not contain any Kurosh element.
\end{theorem}
As a consequence, it generalizes a result in \cite{Herstein1} (see Corollary \ref{generalize Herstein theorem 9}). Moreover,  we have that (see Corollary \ref{conjectures hold locally finite}):
\begin{corollary}
The Conjecture \ref{Bien conjecture}, and hence the Conjecture \ref{Herstein}, holds when $D$ is locally finite.
\end{corollary}

\noindent\textbf{Notations.} In this paper, unless otherwise stated, a division ring $D$ is always assumed to be non-commutative with center $F:=Z(D)$. Whenever we say that $N$ is a normal subgroup of $D$, we mean that $N$ is normal in the multiplicative subgroup $D^*$. All other notations in this paper are standard.
\section{Kurosh elements}
In this section, we will describe some classes of Kurosh elements in an arbitrary division ring. The following Lemma is elementary and the proof may be seen for instant in  {\cite{HDB12}}.
\begin{Lemma}\label{2.4} Let $D$ be a division ring. If $K$ is a subring of $D$ containing $F$ such that $K$ is a finite dimensional vector space over $F$ then $K$ is a (centrally finite) division ring.
\end{Lemma}

\begin{Prop} \label{Kurosh torsion}
Every non-central torsion element of a division ring is a Kurosh element.
\end{Prop}
\begin{proof} Let $D$ be a division ring and $a\notin F$ a non-central torsion element of $D$. Let $n$ be the smallest positive integer such that $a^n=1$. Then the field extention $F(a)/F$ is a finite extention and $1,a,\cdots,a^{n-1}$ are all roots of the equation $x^n=1$. Hence the finite extension $F(a)/F$ is normal. It follows that the Galois group $G:=\textrm{Gal}(F(a)/F)$ is finite and non-trivial. Let $\phi\in G $ be a non-trivial $F$-automorphism of $F(a)$. Then $\phi(a)=a^i$ for some $i>1$. Moreover, $a$ and $\phi(a)$ have the same minimal polynomial over $F$. By Dickson's Theorem {\cite[16.8]{Lam-first-course}}, $a$ and $\phi(a)$ are conjugate in $D$, i.e., there exists an element $u\in D^*$ such that $\phi(a)=uau^{-1}$. For each $m$, $\phi(a^m)=(\phi(a))^m=(uau^{-1})^m=uau^{-1}\cdot uau^{-1}\cdots uau^{-1}=ua^mu^{-1},$ so that for every $x=\sum_m\alpha_ma^m\in F(a)$, one has $\phi(x)=\sum_m\alpha_m\phi(a^m)=\sum_m\alpha_mua^mu^{-1}=u(\sum_m\alpha_ma^m)u^{-1}=uxu^{-1}.$
Thus $\phi^t(x)=u^txu^{-t}$ for every integer $t>1$. Since $G$ is finite, there exists an integer $k>1$ such that $\phi^k=Id_{F(a)}$, and so $x=\phi^k(x)=u^kxu^{-k}$. In particular, $au^k=u^ka$. Set $K:=C_D(u^k)=\left\{ d\in D|\ du^k=u^kd\right\}$. It is easy to check that $K$ is a division subring of $D$ containing $a$ and $u$. Let $F_K:=Z(K)$ be the center of $K$ and $D_1:=F_K[a,u]$ the subring of $K$ generated by $a$ and $u$ over $F_K$. Because $\phi(a)=uau^{-1}=a^i\ne a$,  $ua=a^iu$, and thus the ring $D_1$ is non-commutative. Moreover,  every element $y$ of $D_1$ can be written as the form $y={\alpha _1}{a^{{s_1}}}{u^{{t_1}}} + {\alpha _2}{a^{{s_2}}}{u^{{t_2}}} +  \cdots  + {\alpha _m}{a^{{s_m}}}{u^{{t_m}}}$ where $m, s_i,t_i$ is positive integers, and $\alpha_i\in F_K$. Notice that $a^n=1$ and $u^k\in F_K$, so $D_1$ is a finite dimensional space over $F_K$. It follows from Lemma~\ref{2.4} that $D_1$ is a division ring. Obviously, the center $F_1:=Z(D_1)$ of $D_1$ contains $F_K$. Thus $\dim_{F_1}D_1 \le \dim _{F_K}D_1<\infty$ and $D_1$ is centrally finite. 
\end{proof}

\begin{Prop}\label{Korosh ne 2}
Let $D$ be a  division ring whose center $F$  is a field of characteristic $p\ne 2$. Then every non-central element of $D$ satisfying $a^2\in F$ is a Kurosh element. 
\end{Prop}

\begin{proof} Since $a^2\in F, a\notin F$ and $p=\textrm{char}(F)\ne 2$, the extension $F(a)/F$ is a non-trivial Galois extension. Hence there exists a non-trivial automorphism $\phi\in\textrm{Gal}(F(a)/F)$ such that $\phi(a)=-a\ne a$. We now proceed as in the proof of Proposition \ref{Kurosh torsion} to decude that $a$ is a Kurosh element.
\end{proof}

If, in addition, the division ring $D$ is algebraic, we have the same result even in the case $\textrm{char}(F)=2$. Generally, we have an analogue result when $\textrm{char}(F)$ is an arbitrary prime. Before stating the result, we need the following lemma which can be easily proved by induction (see  \cite{Jacobson}).

\begin{Lemma}\label{Lemma commutator}
For every $x,y\in D$, set $[x,y]=xy-yx$ and $[x,y,y,\cdots,y]=[[x,y,\cdots,y]y]$. Then for each positive integer $n>0$ we have that 
$$[x,\underbrace{y,\cdots,y}_{\textrm{$n$ factors}}]=\sum_{i=0}^n(-1)^i\binom{n}{i}y^ixy^{n-i}.$$
\end{Lemma}
\begin{Th}\label{Kurosh algebraic}
Let $D$ be an algebraic division ring whose center $F$ is a field of positive characteristic $p$. Then every non-central element $a$ of $D$ satifying $a^p\in F$ is a Kurosh element. 
\end{Th}
\begin{proof}
Let $x$ be an element of $D$ such that $xa\ne ax$. We have from Lemma \ref{Lemma commutator} that $$[x,\underbrace{a,\cdots,a}_{\textrm{$p$ factors}}]=\sum_{i=0}^p(-1)^p\binom{p}{i}a^ixa^{n-i}.$$
Since $\textrm{char}(F)=p$, $a^p\in F$. Hence, $[x,\underbrace{a,\cdots,a}_{\textrm{$p$ factors}}]=a^px-xa^p=0$. Call $r\leq p$ the minimal integer such that $[x,\underbrace{a,\cdots,a}_{\textrm{$r$ factors}}]=0$. Because $xa\ne ax$, $r\ge 2$. Set $y:=[x,\underbrace{a,\cdots,a}_{\textrm{$r-2$ factors}}]$. Then $ya-ay=[y,a]\ne 0$ by the minimality of $r$, and $[[y,a],a]=0$, which implies that $ya-ay\in C_D(a)$, so that $ba-ab=a$ with $b=y[y,a]^{-1}a\in D$. Let $K:=F[a,b]$ be the subring of $D$ generated by $F$ and $a,b$. It suffices to show that $K$ is a centrally finite division ring and $a$ does not belong to the center of $K$. Indeed, it is clearly that $a$ is not in the center of $K$ as $ab\ne ba$. Now because $ba= ab+a$, every element in $K$ can be written as the form $\sum_{i,j} \alpha_{i,j}a^ib^j$ where $\alpha_{i,j}\in F$. By the assumption that $D$ is algebraic, there exists $m>0$ such that $\sum_{i,j} \alpha_{i,j}a^ib^j=\sum_{i=1,j=1}^m \alpha_{i,j}a^ib^j$. Therefore, $K$ is a finite dimensional vector space over $F$ whose minimal generating set is a subset of $\{\, a^ib^j\mid 1\le i,j\le m\,\}$. Hence $K$ is a centrally finite division ring by Lemma~\ref{2.4}.
\end{proof}


A subfield $K$ of a division ring $D$ is called {\it maximal} if there is no subfield of $D$ strictly containing $K$. Such a maximal subfield $K$ always exists by  Zorn's Lemma. By \cite[15.8]{Lam-first-course}, $D$ is a centrally finite division ring if and only if $\dim_FK<\infty$. However, $D$ is not neccessary algebraic over $F$ even though $K$ is an algebraic extention of $F$ . In Section 4, we give an example of a non-algebraic division ring $D$ whose maximal field $K$ is algebraic over $F$. The following theorem describes a class of Kurosh elements in such a division ring.
\begin{Th} Let $D$ be a division ring  and $K$ a maximal field of $D$ which is algebraic over $F$. If $K=F(S)$ for some subset $S$ of $K$ then every element $a$ in $D\setminus K$ with $S_a:=\{x\in S|\ ax\ne ax\}$  finite is a Kurosh element.
\end{Th}
\begin{proof} Let $D_1:=K(a)$ be the division subring of $D$ generated by $K$ and $a$. It suffices to prove that $D_1$ is centrally finite and $a$ is not in the center $F_1=Z(D_1)$. Since $K$ is a maximal field of $D$, it is also a maximal field of $D_1$. Since $a\notin K$, it follows that $a\notin F_1$. Obviously, each $x\in F(S\setminus S_a)$ commutes with $a$, and hence belongs to $F_1$. It means that $F_1$ contains $F(S\backslash S_a)$. Because $K$ is algebraic over $F$, $K$ is also algebraic over $F(S\backslash S_a)\supseteq F$. Moreover, by the fact that $S_a$ is finite, $K=F(S\backslash S_a)(S_a)$ is a finitely generated field extention over $F(S\backslash S_a)$, and thus $\dim_{F(S\backslash S_a)}K<\infty$. Therefore, $\dim _{F_1} K\le \dim_{F(S\backslash S_a)}K<\infty$. By {\cite[15.8]{Lam-first-course}}, the ring $D_1$ is centrally finite.  
\end{proof}

\section{The main results}

We first have the following useful lemma, which is a partial result of the Conjecture~\ref{Herstein}. 
\begin{Lemma}\label{radical in center}
Let $D$ be a centrally finite division ring and $N$ a normal subgroup of $D$. If  $N$ is radical over $F$, then $N\subseteq F$.
\end{Lemma}
\begin{proof}
Let $a\in N$ and $x\in D^*$. Then $u:=xax^{-1}a^{-1}\in N$ as $N\trianglelefteq D^*$. Since $N$ is radical over $F$, there exists $n(x,a)>0$ such that $u^{n(x,a)}\in F$. Since $D$ is finite dimensional over $F$, it follows from  {\cite[Sublemma]{Herstein1}} that $u^{n(x,a)}$ is a root of unity. Thus $u^t=1$ for some $t>0$. From {\cite[Theorem 9]{Herstein1}}, it follows that $u\in F$, i.e., $xax^{-1}a^{-1}\in F$. So $xF(a)x^{-1}\subseteq F(a)$ for all $x\in D$. By the Brauer-Cartan-Hua Theorem ({\cite[13.17]{Lam-first-course}}) and $F(a)$ is commutative, $F(a)=F$. Hence $a\in F$.
\end{proof}

\begin{Lemma}\label{xN and y commute}
Let $D$ be a centrally finite division ring  and $N$  a normal subgroup of $D$. If, for any $x,y\in N$, $x$ is radical over $C_D(y)$ then $N\subseteq F$.
\end{Lemma}
\begin{proof} First of all, we prove that $N$ is commutative. Let $x,y\in N$ and $K:=F(x,y)$. Since $N$ is normal in $D^*$, $N\cap K$ is normal in $K^*$. Let $a\in N\cap K$. By hypothesis, there are $n(a,x)$ and $n(a,y)$ such that 
$$a^{n(a,x)}x=xa^{n(a,x)}$$
$$a^{n(a,y)}y=ya^{n(a,y)}.$$
Put $n=n(a,x)n(a,y)$. Then $a^n=a^{n(a,x)n(a,y)}=(xa^{n(a,x)}x^{-1})^{n(a,y)}=xa^nx^{-1}.$ Similarly, $a^n=a^{n(a,y)n(a,x)}=(ya^{n(a,y)}y^{-1})^{n(a,x)}=ya^ny^{-1}.$ These imply that $a^n$ is in the center $Z(K)$ of $K$. So $N\cap K$ is radical over $Z(K)$. By Lemma \ref{radical in center}, $N\cap K$ is contained in $Z(K)$. In particular, $x$ and $y$ commute. Thus $N$ is commutative.

Let $H$ be a division subring of $D$ generated by $N$ over $F$, then $aHa^{-1}\leq H$ for all $a\in N$. If $N\not\subseteq F$ then by \cite{Herstein-Scott63}, either $H\subseteq F$ or $H=D$. Since $N$ is commutative, $H$ is also commutative. Therefore $H\subseteq F$. Hence $N\subseteq F$ as desired.
\end{proof}

Now we can prove the main result stated in the introduction which gives an affirmative answer to the Conjecture~\ref{Bien conjecture} for centrally finite division rings.
\begin{Lemma}\label{generalized Herstein}
Let $D$ be a centrally finite division ring, $K$  a proper division subring of $D$ and $N$ a normal subgroup of $D$. If $N$ is radical over $K$ then $N\subseteq F$.
\end{Lemma}
\begin{proof}
Suppose that $N\not\subseteq F$. If $N\setminus K=\emptyset$, then $N\subseteq K$. By \cite{Herstein-Scott63}, either $K\subseteq F$ or $K=D$. Since $K\ne D$ by the hypothesis, it follows that $K\subseteq F$. Hence $N\subseteq F$, a contradiction. Thus, we may assume that $N\setminus K\ne\emptyset$. In order to prove the lemma, it suffices to show that all elements in $N$ satisfy the conditions of Lemma \ref{xN and y commute}. Let $a,b\in N$.  Assume first that $a\notin K$. 

If $b\notin K$, assume by contradiction that $a$ is not radical over $C_D(b)$, i.e., $a^nb\ne ba^n$ for all $n>0$. Then $a+b\ne 0, a\ne \pm 1$ and $b\ne \pm 1$. Thus $x=(a+b)a(a+b)^{-1}, y=(b+1)a(b+1)^{-1}\in N$
as $N$ is normal in $D^*$. Since $N$ is radical over $K$, there exist $m_x>0$ and $m_y>0$ such that $x^{m_x}\in K$ and $y^{m_y}\in K$. Then $x^m\in K$ and $y^m\in K$ where $m=m_xm_y$. Now
$$x^mb-y^mb+x^ma-y^m=x^m(a+b)-y^m(b+1)=(a+b)a^m-(b+1)a^m=a^m(a-1).$$ This implies $$(x^m-y^m)b=a^m(a-1)+y^m-x^ma.$$ If $(x^m-y^m)\ne 0$, then $b=(x^m-y^m)^{-1}[a^m(a-1)+y^m-x^ma]\in K$, and this contradicts to the choice of $b$. Therefore $x^m=y^m$, and thus $a^m(a-1)=y^m(a-1)$. Since $a\ne 1, a^m=y^m=(b+1)a^m(b+1)^{-1}$ and it follows that $a^mb=ba^m$, which is a contradiction.

 If $b\in K$, consider an element $x\in N\setminus K$. Since $xb\notin K$, by the previous case, there exist positive integers $r,s$ such that $a^rxb=xba^r\quad\textrm{and}\quad a^sx=xa^s.$ These imply that
$$a^{rs}=(xb)^{-1}a^{rs}(xb)=b^{-1}(x^{-1}a^{rs}x)b=b^{-1}(x^{-1}a^{s}x)^rb=b^{-1}a^{rs}b$$
and so $a^{rs}b=ba^{rs}$.

 Assume now that $a\notin K$. Since $N$ is radical over $K$ then $a^m\in K$ for some $m>0$. By the above argument, there exsits $n>0$ such that $a^{mn}b=ba^{mn}$.

As a conclusion, in any case, any two elements $a,b\in N$ satisfy the conditions of Lemma \ref{xN and y commute}. Therefore $N\subseteq F$.
\end{proof}

\begin{Th}\label{main theorem}
Let $D$ be a division ring, $K$ a proper division subring of $D$ containing $F$ and $N$ a normal subgroup of $D$. If $N$ is radical over $K$ then $N\setminus K$ does not contain any Kurosh element.
\end{Th}
\begin{proof}
Assume that $N$ contains a Kurosh element $x\notin K$. Then there is a division subring $D_x$ in $D$ such that $D_x$ contains $x$ and $D_x$ is finite dimensional over its center $F_x$. Let $N_x:=N\cap D_x$ be a normal subgroup of $D_x^*$. Since $N$ is radical over $K$, it follows that $N_x$ is radical over $K_x:=K\cap D_x$. Since $x\notin K$, one has $K_x\ne D_x$. Then by Lemma \ref{generalized Herstein}, $N_x\subseteq F_x$. In particular, $x\in F_x$, which contradicts to the definition of $x$.
\end{proof}

Theorem \ref{main theorem} is a generalization of a result recently obtained in \cite{HDB12}. As a consequence, the Conjecture \ref{Bien conjecture} holds for locally division rings:
\begin{Cor}\label{conjectures hold locally finite}
In a locally finite division ring $D$, every normal subgroup $N$ of $D$ radical over a proper division subring of $D$ is contained in $F$. 
\end{Cor}
\begin{proof}Let $K$ be a proper division subring of $D$ such that $N$ is radical over $K$. Since $D$ is locally finite, every element of $D$ is then Kurosh. By Theorem \ref{main theorem}, $N\subseteq K$. By \cite{Herstein-Scott63}, either $K\subseteq F$ or $K=D$. Since $K\ne D$, we have that $K\subseteq F$. In particular, $N\subseteq F$.
\end{proof}

\begin{Cor}\label{generalize Herstein theorem 9}
Let $D$ be a division ring and $N$ a normal subgroup of $D$. If $N$ is radical over a division subring $K$ of $D$ then $K$ contains all torsion elements of $N$.
\end{Cor}
\begin{proof}
If $x$ is a torsion element of $N$ then $x$ is Kurosh by Theorem \ref{Kurosh torsion}. It follows from Theorem \ref{main theorem} that  $x\in K$.
\end{proof}

\section{An example}
In this section, we will follow the Mal'cev-Neumann construction of Laurent series rings to construct a non-trivial non-algebraic division ring $D$ whose all elements are non-algebraic Kurosh elements.

Let $G=\oplus_{i=1}^\infty\mathbb{Z}$ be the free abelian group of infinite rank. Let us order $G$ lexicographically by the rule that for any $(n_1,n_2,\cdots)$ and $(m_1,m_2,\cdots)$ in $G$, $(n_1,n_2,\cdots)<(m_1,m_2,\cdots)$ if and only if either $n_1<m_1$ or there exists $k$ such that $n_i=m_i$ for all $i=1,\cdots,k-1$ and $n_k<m_k$. Then with this order, $G$ is a totally well-ordered set.

For an increasing infinite sequence of primes $p_1<p_2<\cdots$, set $K:=\mathbb{Q}(\sqrt{p_1},\sqrt{p_2},\cdots)$ to be the subfield of $\mathbb{R}$ generated by $\mathbb{Q}$ and the $\sqrt{p_1}, \sqrt{p_2},\cdots$. Let $G$ act on $K$ by $\mathbb{Q}$-automorphisms as the following: for every $g=(n_1,n_2,\cdots)\in G$, $g$ fixes all rational numbers and for each $i$, $(\sqrt{p_i})^g=(-1)^{n_i}\sqrt{p_i}$.

From now, we write the operation of $G$ multiplicative. Define $D=K((G))$ to be the set of all formal sums $\alpha=\sum_{g\in G}\alpha_gg$ with $\textrm{supp}(\alpha):=\{g\in G : \alpha_g\ne 0 \}$ well-ordered. For $\alpha=\sum_{g\in G}\alpha_g g$ and $\beta=\sum_{g\in G}\beta_gg$, define the operations on $D$ as the following
$$\alpha+\beta =\sum_{g\in G} (\alpha_g+\beta_g)g$$
$$\alpha\cdot\beta=\sum_{g\in G}\left(\sum_{uv=g}\alpha_u\beta_v^u\right)g.$$
\begin{Th}{\cite[14.21]{Lam-first-course}}
The set $D=K((G))$ with two operations defined as above is a division ring.
\end{Th}

We are now going to compute the center $F$ of $D$. Let $H$ be the subgroup of $G$ consisting of all $g^2$ with $g\in G$ and let $\mathbb{Q}((H))$ be the set of all formal sum $\alpha=\sum_{h\in H}\alpha_hh$ with $\alpha_h\in\mathbb{Q}$ and $\textrm{supp}(\alpha)$ well-ordered.
\begin{Prop}\label{central}
The center $F$ of $D$ is $\mathbb{Q}((H))$.
\end{Prop}
\begin{proof} It is easily seen that $g^2\in F$ for every $g\in G$. For $\alpha=\sum_{h\in H}\alpha_hh\in H$ and $\beta=\sum_{g\in G}\beta_gg\in D$, notice that $\beta_g^h=\beta_g$  and $\alpha_h^g=\alpha^h$ for all $h\in H,g\in G$ since $\alpha_h\in\mathbb{Q}$. We have that
$$\alpha\cdot\beta=\sum_{g\in G}\left(\sum_{uv=g}\alpha_u\beta_v^u\right)g=\sum_{g\in G}\left(\sum_{uv=g}\alpha_u\beta_v\right)g$$
$$\beta\cdot\alpha=\sum_{g\in G}\left(\sum_{uv=g}\beta_u\alpha_v^u\right)g=\sum_{g\in G}\left(\sum_{uv=g}\alpha_u\beta_v\right)g$$
which imply that $\alpha\beta=\beta\alpha$. It holds for all $\beta\in G$, thus $\alpha\in F$. Conversely, assume that $\alpha=\sum_{g\in G}\alpha_gg\in F$. Then $\sqrt{p_i}\alpha=\alpha\sqrt{p_i}$ for every $i\geq 1$. This means that $\sqrt{p_i}g=g\sqrt{p_i}=(-1)^{g_i}\sqrt{p_i}g$ for every $g\in\textrm{supp}(\alpha)$. Hence $g_i$ is a square for every $i$, which implies that $g\in H$. Moreover, since $\alpha\in F$, for each $\alpha_g\ne 0$, we have that $\alpha_g\alpha=\alpha\alpha_g$. This implies that $\alpha_gg=g\alpha_g=\alpha_g^gg$. Therefore $\alpha_g$ is fixed by $g$. It follows that $\alpha_g\in\mathbb{Q}$ for all $g\in\textrm{supp}(\alpha)$. Hence $\alpha\in\mathbb{Q}((H))$.
\end{proof}

For each positive integer $i$, let $x_i$ be the element of $G$ with $1$ in the position $i$ and $0$ elsewhere. Then $x_i^{-1}<x_{i+1}^{-1}$. Consider the element $\alpha=x_1^{-1}+x_2^{-1}+\cdots$. Then $$\textrm{supp} (\alpha)=\{\, x_i^{-1}\mid i=1,2,\cdots\,\}$$ is well-ordered. It means $\alpha\in D$.
\begin{Lemma}
$\alpha$ is not algebraic over $F$.
\end{Lemma}
\begin{proof}
Consider the equality
$$a_0+a_1\alpha+a_2\alpha^2+\cdots+a_n\alpha^n=0$$
where $a_i\in F$ for each $i=0,\cdots, n$. The $x=x_1^{-1}\cdots x_n^{-1}$ does not appear in the expressions of $\alpha, \alpha^2,\cdots, \alpha^{n-1}$ and it appears in $\alpha^n$ with coefficients $n\!$. Hence the coefficient of $x$ in the left handside of the above equation is $a_nn!$. It follows that $a_n=0$. By induction, we have that $a_0=a_1=\cdots=a_n=0$. Thus the set $\{1,\alpha,\cdots,\alpha^n\}$ is independent for every $n$. Hence $\alpha$ is not algebraic over $F$.
\end{proof}

\begin{Lemma}
For $\alpha$ is as above, let $D_\alpha:=F(\alpha,\sqrt{p_1},\sqrt{p_2},\cdots)$ be the division subring of $D$ generated by $\alpha$ and the $\sqrt{p_i}$'s. Then the center of $D_\alpha$ equals to $F$ and $K_\alpha:=F(\sqrt{p_1},\sqrt{p_2},\cdots)$ is a maximal field of $D_\alpha$. 
\end{Lemma}
\begin{proof} Using the same arguments as in the proof of Proposition \ref{central}, we obtain that the center of $D_\alpha$ is $F$. 

In order to prove that $K_\alpha$ is a maximal subfield of $D_\alpha$, by {\cite[15.7]{Lam-first-course}}, we only need to prove that $C_{D_\alpha}(K_\alpha)=K_\alpha$. Let $a\in C_{D_\alpha}(K_\alpha)\setminus K_{\alpha}$. Then there exists some $i$ such that $x_i$ appears in the expression of $a$ as a formal sum. Since $x_i^2\in F$, $a$ can be expressed in the form $a=bx_i+c$, where $b\ne 0$ and $x_i$ does not appear in the formal expressions of $b$ and $c$. Therefore $\sqrt{p_i}a-a\sqrt{p_i}=2b\sqrt{p_i}x_i\ne 0$. It follows that $a$ does not commute with $\sqrt{p_i}\in K_\alpha$, which is a contradiction. Hence $K_\alpha$ is a maximal subfield of $D_\alpha$.
\end{proof}

\begin{Th}
The  ring $D_\alpha$ is a non-algebraic division ring whose every non-central element is Kurosh.
\end{Th}
\begin{proof}
For each $z\in D_{\alpha}$, there are only finitely many primes occuring in the expression of $z$, say $p_1,\cdots, p_n$. Then $z\in D_n:=F(\alpha,\sqrt{p_1},\cdots,\sqrt{p_n})$. It is easy to see that $F(\alpha-x_1^{-1}-x_2^{-1}-\cdots-x_n^{-1})$ is contained in the central $F_n$ of $D_n$. Since $$D_n=F(\alpha-x_1^{-1}-x_2^{-1}-\cdots-x_n^{-1})((x_1,\cdots,x_n,\sqrt{p_1},\cdots,\sqrt{p_n})) $$  $$=F_n(x_1,\cdots,x_n,\sqrt{p_1},\cdots,\sqrt{p_n}).$$ Using the fact that every element of $D_n$ can be written as form $$\epsilon x_1^{s_1}x_2^{s_2}\cdots x_n^{s_n}\sqrt{p_1}_1^{t_1}\sqrt{p_2}^{t_2}\cdots \sqrt{p_n}^{t_n},$$ where $\epsilon\in F_n$, $s_1,s_2,\cdots, s_n, t_1, t_2,\cdots, t_n$ are $0$ or $1$. It follows that $D_n$ is centrally finite over its center $F_n$, which implies that $z$ is Kurosh.
\end{proof}

We now provide an example of a division ring $D$ with  a proper subring $H$ and a normal subgroup $N$ of $D$ such that $N$ radical over $H$ but $N$ is not commutative.

Let $D$ be a division ring with  center $F$, $(G,\prec)$ an ordered abelian (multiplicative) group, and $v:D^*\to G$ a valuation of $D^*$; that is, $v$ satisties (1) $v(a.b)=v(a)v(b),$ for any $a,b\in D^*$ and (2) $v(a+b)\ge \min\{v(a),v(b)\}, $ for each $ a,b\in D^*, a+b\ne 0$. Choose $v$ such that there exist $x,y\in D^*$ with $x\notin F, v(x)<1, v(y)<v(x^n), \forall n\in \bb{Z}$.  We get  $D=K((G))$ as in Theorem 4.1. Put $\alpha=\sum_{g\in G}\alpha_gg$, we have  $v(\alpha)=\min\sup(\alpha)$, $x=x_2^{-1},$ and $y=x_1^{-1}$. Put $H=\{a\in D^*\mid \exists n\in \bb{Z}, v(a)\ge v(x^n)\}\cup \{0\}$.
\begin{Th} With the above notations, we have that $H$ is a proper subring of $D$, $H^*$is normal in $D^*$ and $H^*$ is not commutative.
\end{Th}
\begin{proof} We first prove that $H$ is a subring of $D$. For any $a,b\in H$, we have that $v(a,b)=v(a)v(b)\geq v(x^n)v(x^m)$ for some $m,n\in\bb{Z}$. Hence $v(ab)\geq v(x^nx^m)=v(x^{n+m})$, which implies that $ab\in H$. Assume that $a+b\ne 0$. Then $v(a+b)\geq \min\{v(a),v(b)\}\geq\min\{v(x^n),v(x^m)\}$. It follows that $a+b\in H$. Moreover, $v(y)<v(x^n), \forall n\in \bb{Z}$, $y\notin H$ implies that $H$ is a proper subring of $D$. 

Now we check that $H^*$ is normal in $D^*$. For any element $a\in D^*, b\in H$, we have that $v(a^{-1}ba)=v(a^{-1})v(b)v(a)=v(b)$, and so $a^{-1}ba\in H$. Therefore, $a^{-1}Ha\subseteq H$. As a corollary, $a^{-1}Ha=H$. It implies that $H^*=(a^{-1}Ha)^*=aH^*a^{-1}$ which means $H^*$ is normal in $D^*$. Notice that $x_2, x^{-1}_2\in H$, it follows that $x_2\in H^*$. On the other hand, $x_2\notin F$. Hence $H^*$ is not contained in $F$. By \cite[14.4.4]{S88}, $H^*$ is not commutative. 
\end{proof}

\subsection*{Acknowledgment} The authors are very thankful to the referee for carefully reading our paper and making
useful comments and recommendations. In addition, the first author is also extremely grateful for the support given by the Vietnam National Foundation for Science and Technology Development (NAFOSTED).

\end{document}